\documentclass[11pt, twoside]{article}
\usepackage{amsfonts}
\usepackage{ulem}

\usepackage{amssymb}
\usepackage{amsmath}
\usepackage{amsthm}
\usepackage{xcolor}
\usepackage{mathrsfs}

\usepackage{verbatim}

\allowdisplaybreaks

\allowdisplaybreaks

\def\R{{\mathbb R}}
\def\Z{{\mathbb Z}}
\def\N{{\mathbb N}}

\def\omiga{\omega}
\def\huaA{{\mathcal A}}

\def\Smn{S_{m,N}}

\def\jsJ{\left|J\right|}

\def\jssmn{\left|S_{m,N}\right|}
\def\kaf{\mathcal{X}}
\def\pfzy{\frac{1}{p}}

\def\M{\mathcal{M}}
\def\xieg{\setminus}
\def\kdjdz{\left|k\right|}

\def\R{{\mathbb R}}

\def\Z{{\mathbb Z}}

\def\z+{{\mathbb Z}_+}

\def\N{{\mathbb N}}

\def\bint{{\ifinner\rlap{\bf\kern.30em--}
\int\else\rlap{\bf\kern.35em--}\int\fi}\ignorespaces}

\def\sbint{{\ifinner\rlap{\bf\kern.32em--}
\hspace{0.078cm}\int\else\rlap{\bf\kern.45em--}\int\fi}\ignorespaces}

\newtheorem{theorem}{Theorem}[section]
\newtheorem{lemma}[theorem]{Lemma}

\newtheorem{proposition}[theorem]{Proposition}
\theoremstyle{definition}

\newtheorem{remark}[theorem]{Remark}

\newtheorem{definition}[theorem]{Definition}
\numberwithin{equation}{section}

\numberwithin{equation}{section}

\numberwithin{equation}{section}

\begin{document}

\arraycolsep=1pt

\title{\Large\bf Discrete Riesz Potentials on Discrete Weighted Morrey Spaces \footnotetext{\hspace{-0.35cm} {\it 2020
Mathematics Subject Classification}. {42B35, 46B45, 47B06.}
\endgraf{\it Key words and phrases.} weight, discrete Morrey space, discrete fractional Hardy-Littlewood maximal operator, discrete Riesz potential, Whitney decomposition.
\endgraf This project is supported by the National Natural Science Foundation of China (Grant Nos. 12261083 \& 12161083).
\endgraf $^\ast$\,Corresponding author.
}}
\author{Xuebing Hao, Shuai Yang and Baode Li$^\ast$}
\date{ }
\maketitle

\vspace{-0.8cm}

\begin{center}
\begin{minipage}{13cm}\small
{\noindent{\bf Abstract.}
Let $0<\alpha<1$. We obtain the boundedness of the discrete fractional Hardy-Littlewood maximal operators $\M_\alpha$ on discrete weighted Lebesgue spaces. From this and a discrete version of Whitney decomposition theorem, we deduce the boundedness of the discrete Riesz potentials $I_\alpha$ on discrete weighted Lebesgue spaces. The boundedness of $I_\alpha$ on discrete weighted Morrey spaces is further obtained. Moreover, the boundedness of $\M_\alpha$ is also obtained which is new even for unweighted case.}
\end{minipage}
\end{center}

\section{Introduction\label{s1}}
In 1938, the classical Morrey spaces were introduced by Morrey in \cite{dimorrey} to investigate the local behavior of solutions to second order elliptic partial differential equations. In 1987, Chiarenza and Frasca \cite{1987} showed the boundedness of the fractional integral operator on the Morrey spaces. In 2009, Komori and Shirai \cite{jdM} defined a weighted Morrey space and investigated the boundedness of the fractional integral operator on this space.

During the past few years there has been renewed interest in the area of discrete harmonic analysis and then it becomes an active field of research. See also \cite{wiener20, 2021di, fu, lsmyxz18, 2009p, lsHardy22} for related works on
discrete analogues in harmonic analysis.

Let $m=(m_1,\,\cdots,\,m_d)\in \Z^d$, $N \in\N$ and $S_{m,N}:=\{k\in\Z^d: \|k-m\|_\infty\le N\}$, where as usual $\|(m_1,\,\cdots,\,m_d)\|_\infty:=\max\{|m_i|: 1\le i\le d\}$ for every $m\in \Z^d$. Then $\left| S_{m,N} \right|=(2N+1)^d$$-$the cardinality of $S_{m,N}$.
\begin{definition}
Let $1\le p\le q<\infty$. The {\it discrete Morrey space} is defined by $l^p_q:=l^p_q(\Z)$ the set of sequences $x=\{x(k)\}_{k\in\Z}$ taking values in $\R$ such that
\begin{equation*}
\|x\|_{l^p_q}:=\sup\limits_{m\in\Z,N\in\N}|S_{m,N}|^{\frac{1}{q}-\frac{1}{p}}\left(\sum\limits_{k\in S_{m,N}}|x(k)|^p\right)^\pfzy<\infty.
\end{equation*}
\end{definition}

In 2018, Gunawan, Kikianty and Schwanke \cite{lsmyxz18} studied the discrete Morrey spaces and their generalizations, and obtain necessary and sufficient conditions for the inclusion property among these spaces through an estimate for the
characteristic sequences.

In 2019, Gunawan and Schwanke \cite{jdszzlsm} discussed the boundedness of discrete Riesz potentials on discrete Morrey spaces of arbitrary dimension. In addition, more results on discrete Morrey Spaces, we lead the readers to \cite{2022Abe,2022A,2021A,2020H,2019K,2023W}.

\begin{theorem}\label{t1.2}
Let $0<\alpha<1$ and $1<p<q<\frac{d}{\alpha}$. Set $s=\frac{dp}{d-\alpha q}$ and $t=\frac{qs}{p}$. Then $I_\alpha x\in l{^s_t}(\Z^d)$ for every $x\in l{^p_q}(\Z^d)$ and there exists a positive constant $C$ such that
\begin{equation*}
\|I_\alpha x\|_{{l{^s_t}}(\Z^d)}\le C\|x\|_{{l{^p_q}}(\Z^d)}.
\end{equation*}
\end{theorem}

In 2023, we \cite{hao} introduce a discrete version of weighted Morrey spaces and showed the boundedness of the discrete Hardy-Littlewood maximal operator on the discrete weighted Morrey spaces. Thus, we generate the following natural question:

Is it possible to prove Theorem \ref{t1.2} for weighted sequences on the discrete weighted Morrey spaces when $d=1$ ?

Our aim is to give an affirmative answer to the question which promotes the development of the applications of discrete Muckenhoupt weights and discrete Morrey spaces on discrete harmonic analysis.

This article is organized as follows. In Sect.\,\ref{s2}, we recall some preliminaries on discrete weighted Morrey spaces, discrete $\huaA_p$ weights and discrete $\huaA(p,q)$ weights. In Sect.\,\ref{s3.1}, as an application of \cite[Theorem 3.3]{jdszzlsm}, we obtain the boundedness of $\M_\alpha$ on discrete Morrey spaces. In Sect.\,\ref{s3.2}, inspired by the boundedness of fractional Hardy-Littlewood maximal operators on weighted Lebesgue spaces of Muckenhoupt and Wheeden \cite{1974m} and the boundedness of discrete Hardy-Littlewood maximal operators on discrete Legesgue spaces of Pierce \cite{2009p}, we further obtain the boundedness of $\M_\alpha$ on discrete weighted Legesgue spaces. In Sect.\,\ref{s4.0}, we recall the definition of $I_\alpha$ and prove some basic properties such as uniform convergence of $I_\alpha x$ and that $I_\alpha x$ can inherit the monotonicity of $x$ (see Proposition \ref{l3.10}). In Sect.\,\ref{s4.1}, the boundedness of $I_\alpha$ on discrete weighted Lebesgue spaces is obtained via borrowing some ideas from Muckenhoupt \cite{1974m}. It is worth pointing out that Muckenhoupt's approach with the help of Whitney decomposition theorem is no longer applicable in the discrete version. Fortunately, we overcome this difficulty by establishing a decomposition theorem for any integers set (see Lemma \ref{l3.2}), which plays an important role in our proof. As applications, in Sect.\,\ref{s4.2}, we obtain the boundedness of $I_\alpha$ and $\M_\alpha$ on discrete weighted Morrey spaces.

\section*{Notation}

\ \,\ \ $-\Z:$ set of integers;

$-\Z_+:=\{1,\,2,\,3,\,\cdots\}$;

$-\N:=\{0,\,1,\,2,\,\cdots\}$;

$-\lfloor k \rfloor:$ the nearest integer less than or equal to $k$;

$-\kaf_{J}:$ the characteristic function of interval $J$;

$-S_{m,N}:=\{k\in\Z, |k-m|\le N\}$, where $m\in\Z$ and $N\in\N$;

$-\lambda S_{m,N}:=\{k\in\Z, |k-m|\le \lambda N\}$, where $\lambda\in\Z_+$, $m\in\Z$ and $N\in\N$;

$-\lambda\{i_0\}:=\{i_0-(\lambda-1),\,\cdots,\,i_0,\,\cdots,\,i_0+(\lambda-1)\}$,  where $\lambda\in\Z_+$ and $i_0\in\Z$;

$-\omiga(J):=\sum\limits_{k\in J}\omiga(k)$, for any interval $J\subset\Z$;

$-B^c:$ the complementary set of $B$;

$-r':$ the conjugate exponent of $r$, namely, $\frac{1}{r}+\frac{1}{r'}=1$;

$-C:$ a positive constant which is independent of the main parameters, but it may vary from line to line.




\section{Preliminaries \label{s2}}
In this section, we mainly recall some preliminaries on discrete weighted Morrey spaces and discrete Muckenhoupt classes. Let us begin with the definition of discrete weighted Morrey spaces.

Let $m \in \Z$, $N \in\N$ and $S_{m,N}:=\{m-N,\,\dots,\,m,\,\dots,\,m+N\}$. Then $\left| S_{m,N} \right|=2N+1$$-$the cardinality of $S_{m,N}$. A discrete weight $\omiga$ on $\Z$ is a sequence $\omiga = \{\omiga(k)\}_{k\in\Z}$ of positive real numbers.

\begin{definition}\label{d2.1}
Given $1\le p\le q<\infty$.
\begin{itemize}
\item[$\rm(i)$]\cite[Definition 2.1]{hao}
For two discrete weights $\omiga$ and $v$, we define the {\it discrete weighted Morrey space} $l{^p_q}_{(\omiga,v)}:=l{^p_q}_{(\omiga,v)}(\Z)$ to the space of all sequences $x=\{x(k)\}_{k\in\Z}\subset\R$ for which
\begin{equation*}
\|x\|_{l{^p_q}_{(\omiga,v)}}:=\sup_{m\in \Z,N\in\N}(v(S_{m,N}))^{\frac{1}{q}-\frac{1}{p}}\left(\sum_{k\in S_{m,N}}\left|x(k)\right|^p\omiga(k)\right)^\frac{1}{p}<\infty.
\end{equation*}
Particularly, when $p=q$,
\begin{equation*}
\|x\|_{l{^p_q}_{(\omiga,v)}}=\|x\|_{l^p_\omiga}:=\left(\sum\limits_{k\in\Z}|x(k)|^p\omiga(k)\right)^\frac{1}{p}.
\end{equation*}
\item[$\rm(ii)$]\cite[Page 3]{fu}
Let $\omiga$ be a discrete weight. The discrete weighted weak Lebesgue space $l^{p,weak}_\omiga:=l^{p,weak}_\omiga(\Z)$ is the set of sequences $x=\{x(k)\}_{k\in\Z}\subset\R$ such that
\begin{equation*}
\|x\|_{l^{p,weak}_\omiga}:=\sup_{\lambda>0}\lambda\left[\omiga(\{k\in\Z:|x(k)|>\lambda\})\right]^\frac{1}{p}<\infty.
\end{equation*}
\item[$\rm(iii)$]
When $\omiga=v\equiv 1$, $l{^p_q}_{(\omiga,v)}=:l{^p_q}$, $l^p_\omiga=:l^p$ and $l^{p,weak}_\omiga=:l^{p,weak}$.
\end{itemize}
\end{definition}


Now, we recall the definitions of the following weights \cite[Page 3]{wiener20}. By interval $J$, we mean a finite subset of $\Z$ consisting of consecutive integers, i.e., $J=\{a,\,a+1,\,\dots,\,a+n\}$, $a, n\in \Z$, and $\left|J\right|$ stands for its cardinality.
\begin{definition}\label{d2.6}
A discrete weight $\omiga$ is said to belong to the {\it discrete Muckenhoupt class} $\huaA_1:=\huaA_1(\Z)$ if
\begin{equation*}
\|\omiga\|_{\huaA_1(\Z)}:=\sup\limits_{J\subset \Z}\frac{1}{\jsJ}\left(\frac{1}{\inf\limits_{k\in J}\omiga(k)}\sum\limits_{k\in J}\omiga(k)\right)<\infty.
\end{equation*}
For $1<p<\infty$, a discrete weight $\omiga$ is said to belong to the {\it discrete Muckenhoupt class} $\huaA_p:=\huaA_p(\Z)$ if
\begin{equation*}
\|\omiga\|_{\huaA_p(\Z)}:=\sup\limits_{J\subset\Z}\left(\frac{1}{\left|J\right|}\sum_J\omiga\right)\left(\frac{1}
{\left|J\right|}\sum_J\omiga^{\frac{-1}{p-1}}\right)
^{p-1}<\infty,
\end{equation*}
where $\|\omiga\|_{\huaA_p(\Z)}$ denotes the norm of weight $\omiga$ and $J$ is any bounded interval in $\Z$. Define $\huaA_\infty:=\mathop{\cup}\limits_{1\le p<\infty}\huaA_p$.
\end{definition}

Some basic properties of discrete $\huaA_p$ weights are given as follows.

\begin{proposition}\cite[Proposition 2.8]{hao}\label{m2.8}
Let $\omiga$ be a discrete weight on $\Z$. Then the following statements are equivalent:
\begin{itemize}
\item[\rm(i)] $\omiga\in\huaA_p$ $(1<p<\infty)$;
\item[\rm(ii)] $\frac{1}{\left|J\right|}\sum\limits_{k\in J}\left|x(k)\right|\le C\left(\frac{1}{\omiga(J)}\sum\limits_{k\in J}\left|x(k)\right|^p\omiga(k)\right)^{\frac{1}{p}}$, $x\in l{^p_{\omiga}}(\Z)$.
\end{itemize}
\end{proposition}

\begin{proposition}\cite[Lemma 2.1]{fu}\label{p2.1}
Let $\omiga$ be a discrete weight on $\Z$. If $1\le r\le p<\infty$, then $\huaA_1\subset\huaA_r\subset\huaA_p$.
\end{proposition}

By Propositions \ref{m2.8}$\rm(ii)$ and \ref{p2.1}, we immediately obtain the following Proposition \ref{m2.11}, the proof of which in continuous version can be found in \cite[Page 22]{lu}. Same line of proof also works here.

\begin{proposition}\label{m2.11}
If $\omiga\in \huaA_\infty$, then for any $\epsilon>0$, there exists $\delta\in(0,1)$ such that for any subset $S\subset\,\text{interval}\,J\subset\Z$ with $\left|S\right|\le \delta \left|J\right|$, it holds true that $\omiga(S)\le\epsilon\omiga(J)$.
\end{proposition}

%
%
%

Referring to the definition of $\huaA(p,q)$ on $\R$ (see \cite[Page 139]{lu}), we can give the definition of discrete $\huaA(p,q)$ on $\Z$.

\begin{definition}\label{d2.12}
A discrete weight $\omiga$ is said to belong to $\huaA(p,q)$ on $\Z$ for $1<p,q<\infty$, if the inequality
\begin{equation*}
\|\omiga\|_{\huaA(p,q)(\Z)}:=\sup\limits_{J\subset\Z}\left(\frac{1}{\left|J\right|}\sum_{J}\omiga^q\right)^{\frac{1}{q}}\left(\frac{1}{\left|J\right|}\sum_{J}
\omiga^{-p'}\right)^{\frac{1}{p'}}<\infty,
\end{equation*}
where $\|\omiga\|_{\huaA(p,q)(\Z)}$ denotes the norm of weight $\omiga$ and $J$ is any bounded interval in $\Z$.
\end{definition}

\begin{remark}\label{m2.10}
Similar to the definition of discrete weight $\huaA(p,q)(\Z)$ ($1<p,q<\infty$), replacing $\Z$ by $\N$, we can also give the definition of discrete weight $\huaA(p,q)(\N)$. Particularly, if $\omiga\in\huaA(p,q)(\N)$, then $\widetilde{\omiga}(\cdot):=\omiga(\left|\cdot\right|)\in\huaA(p,q)(\Z)$ and its proof is similar to that of \cite[Proposition 2.11]{hao}.
\end{remark}

By the definitions of $\huaA_p$ and $\huaA(p,q)$, and discrete H\"older's inequality, we immediately obtain the following proposition. The details being omitted and its corresponding results on real line can be referred to \cite[Page 139]{lu}.

\begin{proposition}(Relation between discrete $\huaA(p,q)$ and discrete $\huaA_p$)\label{m2.13}
Suppose that $0<\alpha<1$, $1<p<\frac{1}{\alpha}$ and $\frac{1}{q}=\frac{1}{p}-\alpha$.
\begin{itemize}
\item[\rm(i)] If $p>1$, then $\omiga\in \huaA(p,q)\iff\omiga^q\in \huaA_{1+\frac{q}{p'}}\iff\omiga^{-p'}\in \huaA_{1+\frac{p'}{q}}$.
\item[\rm(ii)] If $p>1$, then $\omiga\in \huaA(p,q)\iff\omiga^q\in \huaA_q\quad and\quad \omiga^p\in \huaA_p$.
\end{itemize}
\end{proposition}

\begin{remark}
By \cite[Page 8]{wiener20} and Remark \ref{m2.10}, we obtain the relationships among weights $\huaA(p,q)(\R_+)$, $\huaA(p,q)(\N)$ and $\huaA(p,q)(\Z)$:
\begin{equation*}\begin{array}{rcl}
\mu\in\huaA(p,q)(\R_{+})\iff\{\mu(k)\}^\infty_{k=0}\in \huaA(p,q)(\N)\Rightarrow\{\mu(\kdjdz)\}_{k\in\Z}\in\huaA(p,q)(\Z).
\end{array}\end{equation*}
\end{remark}

\section{The estimates for discrete fractional Hardy-Littlewood maximal operators \label{s3}}
\subsection{Estimates for discrete fractional Hardy-Littlewood maximal operators on discrete Morrey spaces \label{s3.1}}

\begin{definition}\cite[Definition 2.1]{bzb}
Let $0\le\alpha<1$ and $x=\{x(k)\}_{k\in\Z}\subset\R$ be a sequence. We define the {\it discrete fractional maximal operator} $\M_{\alpha}$ by
\begin{equation*}
\M_{\alpha}x(m):=\sup_{N\in \N}\frac{1}{\jssmn^{1-\alpha}}\sum_{k\in \Smn}\left|x(k)\right|,\quad m\in\Z.
\end{equation*}
\end{definition}

The boundedness of fractional Hardy-Littlewood maximal operators on discrete Morrey spaces is shown in the following Theorem \ref{t3.7}.

\begin{theorem}\label{t3.7}
Let $x=\{x(k)\}_{k\in\Z}$ be a sequence.
\begin{itemize}
\item[\rm(i)] Let $0<\alpha<1$ and $1<p\le q<\frac{1}{\alpha}$. Set $s=\frac{p}{1-\alpha q}$ and $t=\frac{qs}{p}$. Then $\M_\alpha x\in l{^s_t}$ for every $x\in l{^p_q}$ and there exists a positive constant $C$ such that
\begin{equation*}
\|\M_\alpha x\|_{{l{^s_t}}}\le C\|x\|_{{l{^p_q}}}.
\end{equation*}
\item[\rm(ii)] If $x\in l^1$, then $\M_\alpha x\in l^{{\frac{1}{1-\alpha}},weak}$ and for every $\lambda>0$, there exists a positive constant $C$ such that for every $x\in l^1$,
\begin{equation*}
\left|\{k\in \Z:\M_\alpha x(k)>\lambda\}\right|\le C\left(\frac{\|x\|_{l_1}}{\lambda}\right)^\frac{1}{1-\alpha}.
\end{equation*}
\end{itemize}
\end{theorem}

To prove Theorem \ref{t3.7}, we need some lemmas.

The following Lemma \ref{t3.4} is the boundedness of the discrete Riesz potentials on discrete Morrey space of arbitrary dimension $d$. We only use the case of $d=1$.
\begin{lemma}\cite[Theorem 3.3]{jdszzlsm}\label{t3.4}
Let $0<\alpha<1$ and $1<p<q<\frac{1}{\alpha}$. Set $s=\frac{p}{1-\alpha q}$ and $t=\frac{qs}{p}$. Then $I_\alpha x\in l{^s_t}$ for every $x\in l{^p_q}$ and there exists a positive constant $C$ such that
\begin{equation*}
\|I_\alpha x\|_{{l{^s_t}}}\le C\|x\|_{{l{^p_q}}}.
\end{equation*}
\end{lemma}

\begin{lemma}\label{t3.5}\cite[Theorem 1]{lsHardy22}
Let $x=\{x(k)\}_{k\in\Z}$ be a sequence.
\begin{itemize}
\item[\rm(i)] If $x\in l^1$, then $\M x\in l^{1,weak}$ and for every $\lambda>0$,
\begin{equation*}
\left|\{k: \M x(k)>\lambda\}\right|\le \frac{C}{\lambda}\|x\|_{l_1},
\end{equation*}
where C is a positive constant independent of $\alpha$ and $b$.
\item[\rm(ii)] If $x\in l^p$, $1<p\le \infty$, then $\M x\in l^p$ and
\begin{equation*}
\|\M x\|_{l^p}\le C_p\|x\|_{l^p},
\end{equation*}
where $C_p$ is a positive constant depending on p.
\end{itemize}
\end{lemma}

By Lemmas \ref{t3.4} and \ref{t3.5}, we can obtain the following Lemma \ref{c3.5}.

\begin{lemma}\label{c3.5}
Let $0<\alpha<1$, $1\le p\le q<\frac{1}{\alpha}$ and $x=\{x(k)\}_{k\in\Z}$ be a sequence.
\begin{itemize}
\item[\rm(i)] If $x\in l^p_q$, then there exists a positive constant $C$ such that
           \begin{equation*}
           \left|I_\alpha x(k)\right|\le C(\M x(k))^{1-\alpha q}\|x\|{^{\alpha q}_{l{^p_q}}}.
           \end{equation*}
\item[\rm(ii)] If $x\in l^1$, then $I_\alpha x\in l^{{\frac{1}{1-\alpha}},weak}$ and for every $\lambda>0$, there exists a positive constant $C$ such that for every $x\in l^1$,
\begin{equation*}
\left|\{k\in \Z: \left|I_\alpha x(k)\right|>\lambda\}\right|\le C\left(\frac{\|x\|_{l_1}}{\lambda}\right)^\frac{1}{1-\alpha}.
\end{equation*}
\end{itemize}
\end{lemma}

\begin{proof}
By \cite[Theorem 3.3]{jdszzlsm} we know that $\rm(i)$ holds true for $0<\alpha<1$ and $1<p<q<\frac{1}{\alpha}$. Actually, $\rm(i)$ is also true for $0<\alpha<1$ and $1\le p\le q<\frac{1}{\alpha}$ by checking the proof of \cite[Theorem 3.3]{jdszzlsm} with standard modifications. From $\rm(i)$ with $p=q=1$ and Lemma \ref{t3.5}$\rm(i)$, it follows that
\begin{equation*}
\begin{aligned}
\left|\left\{k\in \Z:\left|I_\alpha x(k)\right|>\lambda\right\}\right|&\le \left|\left\{k\in \Z :C\|x\|_{l_1}^\alpha(\M x(k))^{1-\alpha}>\lambda\right\}\right|\\
&=\left|\left\{k\in \Z:\M x(k)>\left(\frac{\lambda}{C\|x\|_{l_1}^\alpha}\right)^{\frac{1}{1-\alpha}}\right\}\right|\\
&\le C\left(\frac{\|x\|{^\alpha_{l_1}}}{\lambda}\right)^\frac{1}{1-\alpha}\|x\|_{l_1}\\
&= C\left(\frac{\|x\|_{l_1}}{\lambda}\right)^\frac{1}{1-\alpha}.
\end{aligned}
\end{equation*}
We finish the proof of Lemma \ref{c3.5}.
\end{proof}

The following lemma illustrates that discrete fractional Hardy-Littlewood maximal operator $\M_\alpha$ can be dominated by discrete Riesz potentials $I_\alpha$.
\begin{lemma}\label{t3.6}
Let $0<\alpha<1$ and $x=\{x(k)\}_{k\in\Z}$ be a sequence. Then $\M_\alpha x(k)\le I_\alpha(\left|x\right|)(k)$ for every $k\in\Z$.
\end{lemma}

\begin{proof}
For fixed $k\in \Z$ and any $N\in\N$, we have
\begin{equation}\label{eq3.6}
\begin{aligned}
I_\alpha(\left|x\right|)(k)&=\sum\limits_{i\in\Z\xieg\{k\}}\frac{\left|x(i)\right|}{\left|k-i\right|^{1-\alpha}}\ge \sum\limits_{i\in S_{k,N+1}}\frac{\left|x(i)\right|}{\left|k-i\right|^{1-\alpha}}\ge \frac{1}{(N+1)^{1-\alpha}}\sum\limits_{i\in S_{k,N+1}}\left|x(i)\right|\\
&\ge \frac{1}{(2N+1)^{1-\alpha}}\sum\limits_{i\in S_{k,N+1}}\left|x(i)\right|.
\end{aligned}\end{equation}
Taking the supremum for any $N\in\N$ on both sides of (\ref{eq3.6}), we obtain
\begin{equation*}
I_\alpha(\left|x\right|)(k)\ge\sup\limits_{N\in\N}\frac{1}{(2N+1)^{1-\alpha}}\sum\limits_{i\in S_{k,N}}\left|x(i)\right|=\M_\alpha x(k).
\end{equation*}
We finish the proof of Lemma \ref{t3.6}.
\end{proof}

\begin{proof}[Proof of Theorem \ref{t3.7}]
By Lemmas \ref{t3.6}, \ref{t3.4} and \ref{c3.5}$\rm(ii)$, we immediately obtain Theorem \ref{t3.7}.
\end{proof}

\subsection{Estimates for discrete fractional Hardy-Littlewood maximal operators on discrete weighted Lebesgue spaces \label{s3.2}}
To discuss the weighted estimates of the discrete Riesz potential $I_\alpha$, we need to consider the weighted estimate of the discrete fractional Hardy-Littlewood maximal operator $\M_\alpha$.

\begin{theorem}\label{t3.1}
Let $0<\alpha<1$, $1<p<\frac{1}{\alpha}$ and $\frac{1}{q}=\frac{1}{p}-\alpha$. If $\omiga\in \huaA(p,q)$ and $x\in l^p_{\omiga^p}$, then there exists a positive constant $C$ such that $\M_\alpha x\in l^q_{\omiga^q}$ and $\|\M_\alpha x\|_{l^q_{\omiga^q}}\le C\|x\|_{l^p_{\omiga^p}}$.
\end{theorem}

In order to prove Theorem \ref{t3.1}, we need some lemmas for preparation. 

\begin{lemma}\label{t3.8}
Let $0<\alpha<1$, $1<p<\frac{1}{\alpha}$ and $\frac{1}{q}=\frac{1}{p}-\alpha$. If $\omiga\in\huaA(p,q)$ and $x\in l^p_{\omiga^p}$, then $\M_\alpha x\in l^{q,weak}_{\omiga^q}$ and for every $\lambda>0$, there exists a positive constant $C$ such that for every $x\in l^p_{\omiga^p}$,
\begin{equation*}
\left(\sum_{\{k\in\Z:\M_\alpha x(k)>\lambda\}}\omiga(k)^q\right)^\frac{1}{q}\le \frac{C}{\lambda}\left(\sum_{k\in\Z}\left|x(k)\omiga(k)\right|^p\right)^\frac{1}{p}.
\end{equation*}
\end{lemma}

\begin{proof}
For any $0<\epsilon<1$, $\lambda>0$, $M\in\Z_+$, $N\in\N$ and $k\in\Z$, we define
\begin{equation*}
\begin{aligned}
& S_{k,N}=\Z\cap Q_{k,N+\epsilon},\,\, where\,\, Q_{k,N+\epsilon}:=\{y\in\R:\left|y-k\right|<N+\epsilon\};\\
& E_\lambda:=\{k\in\Z:\M_\alpha x(k)>\lambda\};\\
& E_{\lambda,M}:=E_\lambda\cap S_{0,M},\,\,where\,\, S_{0,M}:=\{m\in \Z:\left|m-0\right|\le M\}.
\end{aligned}
\end{equation*}
Thus, for every $k\in E_{\lambda,M}$, by the definition of $\M_\alpha$, there exists a $S_{k,N_k}$ such that
\begin{equation}\label{3.6}
\left|S_{k,N_k}\right|^{-1+\alpha}\sum_{m\in S_{k,N_k}}\left|x(m)\right|>\lambda.
\end{equation}
Since $E_{\lambda,M}\subset\mathop{\cup}\limits_{k\in E_{\lambda,M}}S_{k,N_k}\subset\mathop{\cup}\limits_{k\in E_{\lambda,M}}Q_{k,{N_k}+\epsilon}$, by Besicovitch overlapping theorem (see \cite[Page 220]{norm}), there exists $\{k_j\}\subset E_{\lambda,M}$ such that $E_{\lambda,M}\subset \mathop{\cup}\limits_j Q_{k_j,N_j+\epsilon}$ and $\sum\limits_j\kaf_{Q_{k_j,N_j+\epsilon}}(k)\le 4$. Then we further have
\begin{equation}\label{eq1.7}
E_{\lambda,M}\subset\left(\mathop{\bigcup}\limits_j Q_{k_j,N_j+\epsilon}\right)\bigcap\Z=\mathop{\bigcup}\limits_j S_{k_j,N_j},
\end{equation}
\begin{equation}\label{eq2.1}
\sum\limits_j\kaf_{S_{k_j,N_j}}(k)\le\sum\limits_j\kaf_{Q_{k_j,N_j+\epsilon}}(k)\le 4.
\end{equation}
From (\ref{eq1.7}) and $\frac{p}{q}<1$, it follows that
\begin{small}
\begin{equation}\label{3.7}
\begin{aligned}
\left(\sum_{m\in E_{\lambda,M}}\omiga(m)^q\right)^\frac{p}{q}\le &\left(\sum\limits_j\sum\limits_{m\in S_{k_j,N_j}}\omiga(m)^q\right)^\frac{p}{q}\le\sum\limits_j\left(\sum\limits_{m\in S_{k_j,N_j}}\omiga(m)^q\right)^\frac{p}{q}.
\end{aligned}
\end{equation}
\end{small}
By (\ref{3.7}) and (\ref{3.6}), we have
\begin{small}
\begin{equation*}
\left(\sum\limits_{m\in E_{\lambda,M}}\omiga(m)^q\right)^\frac{p}{q}\le \sum_j\left(\sum_{m\in S_{k_j,N_j}}\omiga(m)^q\right)^\frac{p}{q}\left(\frac{1}{\lambda\left|S_{k_j,N_j}\right|^{1-\alpha}}
\sum_{m\in S_{k_j,N_j}}\left|x(m)\right|\right)^p
\end{equation*}
\end{small}
By this, discrete H\"older's inequality, $\omiga\in\huaA(p,q)$ and (\ref{eq2.1}), we obtain
\begin{small}
\begin{equation*}
\begin{aligned}
&\left(\sum\limits_{m\in E_{\lambda,M}}\omiga(m)^q\right)^\frac{p}{q}\\
\quad\le &\sum\limits_j\left(\sum\limits_{m\in S_{k_j,N_j}}\omiga(m)^q\right)^\frac{p}{q}\lambda^{-p}\left|S_{k_j,N_j}\right|^{1-p-\frac{p}{q}}\left(\sum\limits_{m\in S_{k_j,N_j}}\left|x(m)\omiga(m)\right|^p\right)\left(\sum\limits_{m\in S_{k_j,N_j}}\omiga(m)^{-p'}\right)^\frac{p}{p'}\\
\le& C\lambda^{-p}\sum\limits_j\left(\sum\limits_{m\in S_{k_j,N_j}}\omiga(m)^q\right)^\frac{p}{q}\left(\sum\limits_{m\in S_{k_j,N_j}}\left|x(m)\omiga(m)\right|^p\right)\left(\sum\limits_{m\in S_{k_j,N_j}}\omiga(m)^q\right)^{-\frac{p}{q}}\\
=& C\lambda^{-p}\sum\limits_j\sum\limits_{m\in\Z}\left|x(m)\omiga(m)\right|^p\kaf_{S_{k_j,N_j}}(m)\\
\le & C\lambda^{-p}\sum\limits_{m\in\Z}\left|x(m)\omiga(m)\right|^p,\\
\end{aligned}
\end{equation*}
\end{small}
and letting $k\rightarrow +\infty$ on both sides of above inequality, we finish the proof of Theorem \ref{t3.8}.
\end{proof}

\begin{lemma}\cite[Lemma 4.4]{hao}\label{l3.11}
Let $x\in l{^p_\omiga}$ and $p>0$. Then
\begin{equation*}
\|x\|^p_{l^p_\omiga}=p\int^\infty_0 \lambda^{p-1}\sum\limits_{{\{k:\left|x(k)\right|>\lambda}\}}\omiga(k)d\lambda.
\end{equation*}
\end{lemma}

By Lemma \ref{l3.11}, we can obtain the following Lemma \ref{l3.17}, the proof of which in continuous version can be found in \cite[Page 272-274]{1970stein}. Same line of proof also works here.

\begin{lemma}\label{l3.17}
Let $u$ and $v$ be two discrete weights on $\Z$, and let the sublinear operator $T$ be both of weak type $(p_0,q_0)$ and $(p_1,q_1)$ on $\Z$ for $1\le p_i\le q_i\le\infty$, $i=0,1$ and $q_0\neq q_1$. That is, there exists positive constants $C_i$, $i=0,1$ such that for any $\lambda>0$ and $x$,
\begin{equation*}
v(\{k\in\Z:|Tx(k)|>\lambda\})\le \left(\frac{C_i}{\lambda}\|x\|_{l^{p_i}_u}\right)^{q_i},\quad i=0,1.
\end{equation*}
Then
\begin{equation*}
\|Tx\|_{l^{q_t}_v}\le C_t\|x\|_{l^{p_t}_u},\quad t\in (0,1),\quad\frac{1}{p_t}=\frac{1-t}{p_0}+\frac{t}{p_1},\quad\frac{1}{q_t}=\frac{1-t}{q_0}+\frac{t}{q_1},
\end{equation*}
where $C_t\le KC^{1-t}_0C^t_1$, $K=K(p_0,q_0,p_1,q_1,t)$ and when $t\rightarrow 0$ or $1$, $K\rightarrow +\infty$.
\end{lemma}

\begin{lemma}\cite[Proposition 2.15]{hao}\label{l3.1}
If $\omiga\in\huaA_p$ $(1<p<\infty)$, then there exists a constant $\epsilon >0$ such that $p-\epsilon>1$ and $\omiga\in\huaA_{p-\epsilon}$.
\end{lemma}

\begin{proof}[Proof of Theorem \ref{t3.1}]
This proof are similar to that of the continuous version (see the proof of \cite[Theorem 3]{1974m}) by using Lemmas \ref{l3.1}, \ref{t3.8} and \ref{l3.17}. The details being omitted.
\end{proof}

\section{Discrete Riesz potentials \label{s4}}
\subsection{Definition and basic properties of discrete Riesz potentials \label{s4.0}}
\begin{definition}\cite[Page 8]{jdszzlsm}\label{d2.17}
Let $0<\alpha<1$ and $x=\{x(k)\}_{k\in\Z}\subset\R$ be a sequence. The {\it discrete Riesz potential} $I_{\alpha}$ is defined by
\begin{equation*}
I_{\alpha}x(k):=\sum\limits_{i\in\Z\xieg\{k\}}\frac{x(i)}{\left|k-i\right|^{1-\alpha}},\quad k\in\Z.
\end{equation*}
\end{definition}

\begin{proposition}\label{l3.15}
Let $0<\alpha<1$, $1\le p<\frac{1}{\alpha}$ and $x=\{x(k)\}_{k\in\Z}\subset\R$ be a sequence.
\begin{itemize}
\item[\rm(i)]
If $x\in l^p$, then the series $\sum\limits_{i\neq k}\frac{x(i)}{\left|k-i\right|^{1-\alpha}}$ absolutely and uniformly converges on $\Z$.
\item[\rm(ii)]
If $\frac{1}{q}=\frac{1}{p}-\alpha$, $q<2p$ and $x\in l^p_q$, then the series $\sum\limits_{i\neq k}\frac{x(i)}{\left|k-i\right|^{1-\alpha}}$ absolutely and uniformly converges on $\Z$.
\end{itemize}
\end{proposition}

\begin{proof}
$\rm(i)$ For every $i\in\Z$ with $i\neq k$, by H\"older's inequality and $\alpha-\frac{1}{p}<0$, we have
\begin{align*}
\sum\limits_{i\neq k}\frac{\left|x(i)\right|}{\left|k-i\right|^{1-\alpha}}&= \sum\limits^\infty_{j=1}\sum\limits_{2^{j-1}\le\left|k-i\right|<2^j}\frac{\left|x(i)\right|}{\left|k-i\right|^{1-\alpha}}\le \sum\limits^\infty_{j=1}\frac{1}{(2^{j-1})^{1-\alpha}}\sum\limits_{\left|k-i\right|<2^j}\left|x(i)\right|\\
&\le \sum\limits^\infty_{j=1}\frac{1}{(2^{j-1})^{1-\alpha}}\left(\sum\limits_{\left|k-i\right|<2^j}\left|x(i)\right|^p\right)^\pfzy\left(\sum\limits_{\left|k-i\right|<2^j}1\right)^{1-\frac{1}{p}}\\
&\le \|x\|_{l^p}\sum\limits^\infty_{j=1}(2^{j-1})^{\alpha-1}\left(2^{j+1}-1\right)^{1-\frac{1}{p}}\le 4\|x\|_{l^p}\sum\limits^\infty_{j=1}(2^{j+1})^{\alpha-\frac{1}{p}}\\
&\le \frac{4\|x\|_{l^p}}{1-2^{\alpha-\frac{1}{p}}}.
\end{align*}
The proof of $\rm(ii)$ is similar to that of $\rm(i)$, by H\"older's inequality, $\frac{1}{q}=\frac{1}{p}-\alpha$ and $q<2p$, we have
\begin{align*}
\sum\limits_{i\neq k}\frac{\left|x(i)\right|}{\left|k-i\right|^{1-\alpha}}&=\sum\limits^\infty_{j=1}\sum\limits_{2^{j-1}\le\left|k-i\right|<2^j}\frac{\left|x(i)\right|}{\left|k-i\right|^{1-\alpha}}
\le\sum\limits^\infty_{j=1}\frac{1}{(2^{j-1})^{1-\alpha}}\sum\limits_{\left|k-i\right|<2^j}\left|x(i)\right|\\
&\le\sum\limits^\infty_{j=1}\frac{1}{(2^{j-1})^{1-\alpha}}\left(\sum\limits_{\left|k-i\right|<2^j}\left|x(i)\right|^p\right)^\pfzy\left(\sum\limits_{\left|k-i\right|<2^j}1\right)^{1-\frac{1}{p}}\\
&\le\|x\|_{l^p_q}\sum\limits^\infty_{j=1}(2^{j-1})^{\alpha-1}\left(2^{j+1}-1\right)^{1-\frac{1}{q}}\le 4\|x\|_{l^p_q}\sum\limits^\infty_{j=1}(2^{j+1})^{\frac{1}{p}-\frac{2}{q}}\\
&\le\frac{4\|x\|_{l^p_q}}{1-2^{\frac{1}{p}-\frac{2}{q}}}.
\end{align*}
We finish the proof of Proposition \ref{l3.15}.
\end{proof}

Proposition \ref{l3.10} shows that if $x\in l^p$ is monotonic sequence, then $I_\alpha x$ inherits the monotonicity of $x$.

\begin{proposition}\label{l3.10}
Let $0<\alpha<1$, $1\le p<\frac{1}{\alpha}$ and $x=\{x(k)\}_{k\in\Z}\subset\R$ be a monotonic sequence and $x\in l^p$. Then $I_\alpha x$ is also monotonic.
\end{proposition}

\begin{proof}
If $x\in l^p$, then by Proposition \ref{l3.15}, we obtain that $I_\alpha x$ is converges absolutely on $\Z$. Thus,
\begin{small}
\begin{equation*}
\begin{aligned}
I_\alpha x(k+1)&=\sum\limits_{i\in\Z\xieg\{k+1\}}\frac{x(i)}{|i-(k+1)|^{1-\alpha}}\\
&=\cdots +\frac{x(k-3)}{|(k-3)-(k+1)|^{1-\alpha}}+\frac{x(k-2)}{|(k-2)-(k+1)|^{1-\alpha}}+\frac{x(k-1)}{|(k-1)-(k+1)|^{1-\alpha}}\\
&\quad +\frac{x(k)}{|k-(k+1)|^{1-\alpha}}+\frac{x(k+2)}{|(k+2)-(k+1)|^{1-\alpha}}+\frac{x(k+3)}{|(k+3)-(k+1)|^{1-\alpha}}+\cdots\\
&=\cdots +\frac{x(k-3)}{4^{1-\alpha}}+\frac{x(k-2)}{3^{1-\alpha}}+\frac{x(k-1)}{2^{1-\alpha}}+\frac{x(k)}{1^{1-\alpha}}
+\frac{x(k+2)}{1^{1-\alpha}}+\frac{x(k+3)}{2^{1-\alpha}}+\cdots\\
&=\frac{x(k)+x(k+2)}{1^{1-\alpha}}+\frac{x(k-1)+x(k+3)}{2^{1-\alpha}}+\frac{x(k-2)+x(k+4)}{3^{1-\alpha}}+\cdots\\
&=\sum\limits^{\infty}_{j=1}\frac{x(k+1-j)+x(k+1+j)}{j^{1-\alpha}},
\end{aligned}
\end{equation*}
\end{small}
and repeat the above steps, we obtain
$I_\alpha x(k)=\sum\limits_{i\in\Z\xieg\{k\}}\frac{x(i)}{|i-k|^{1-\alpha}}=\sum\limits^{\infty}_{j=1}\frac{x(k-j)+x(k+j)}{j^{1-\alpha}}$. From this, it follows that
\begin{small}
\begin{align*}
\Delta(I_\alpha x(k)) &=I_\alpha x(k+1)-I_\alpha x(k)\\
&=\sum\limits_{j=1}^\infty\frac{x(k+1-j)+x(k+1+j)}{j^{1-\alpha}}-\sum\limits_{j=1}^\infty\frac{x(k-j)+x(k+j)}{j^{1-\alpha}}\\
&=\sum\limits_{j=1}^\infty\bigg(\frac{x(k+1-j)+x(k+1+j)}{j^{1-\alpha}}-\frac{x(k-j)+x(k+j)}{j^{1-\alpha}}\bigg)\\
&=\sum\limits_{j=1}^\infty\bigg(\frac{x(k+1-j)-x(k-j)}{j^{1-\alpha}}+\frac{x(k+1+j)-x(k+j)}{j^{1-\alpha}}\bigg)\\
&=\sum\limits_{j=1}^\infty\bigg(\frac{x(k+1-j)-x(k-j)}{j^{1-\alpha}}\bigg)+\sum\limits_{j=1}^\infty\bigg(\frac{x(k+1+j)-x(k+j)}{j^{1-\alpha}}\bigg).
\end{align*}
\end{small}
If $x$ is nonincreasing, then we have $\Delta(I_\alpha x(k))\le 0$. If $x$ is nondecreasing, then we have $\Delta(I_\alpha x(k))\ge 0$. Therefore, $x$ and $I_\alpha x$ have the same monotonicity. We finish the proof of Proposition \ref{l3.10}.
\end{proof}

\subsection{Estimates for discrete Riesz potentials on discrete weighted Lebesgue spaces \label{s4.1}}

\begin{theorem}\label{t1.1}
Let $0<\alpha<1$, $1< p<\frac{1}{\alpha}$ and $\frac{1}{q}=\frac{1}{p}-\alpha$. If $\omiga(k)\in\huaA(p,q)$ and $x\in l^p_{\omiga^p}$, then $I_\alpha x\in l^q_{\omiga^q}$  and there exists a positive constant $C$ such that
\begin{equation}\label{eq1.3}
\bigg(\sum_{k\in\Z}\left|I_\alpha x(k)\omiga(k)\right|^q\bigg)^\frac{1}{q}\le C\bigg(\sum_{k\in\Z}\left|x(k)\omiga(k)\right|^p\bigg)^\frac{1}{p}.
\end{equation}
\end{theorem}

Now, we use the methods of Muckenhoupt and Wheeden \cite{1974m} to prove Theorem \ref{t1.1}, and we
additionally require the following Lemma \ref{l3.2} which can be seen as a discrete version of Whitney decomposition theorem for sets.

\begin{lemma}(Whitney decomposition theorem)\label{l3.2}
Let $E$ be a set of non-empty integers. Then there exists a list of disjoint symmetric integers intervals $\{S_{m_j,N_j}\}_j$, which satisfies
\begin{equation*}
E=\mathop{\bigcup}\limits_jS_{m_j,N_j}\quad and\quad 4S_{m_j,N_j}\bigcap E^c\neq\varnothing.
\end{equation*}
\end{lemma}

\begin{proof}
Any set of non-empty integers $E$ can be written as the disjoint union of at most countable $E_1$ sets, countable $E_2$ sets, one $E_3$ set and one $E_4$ set, where $E_1$ set is a single integer set like $\{i_0\}$ satisfying $i_0\in E$ but $i_0-1,i_0+1\notin E$, $E_2$ set is the finitely continuous integers set like $\{n,\,n+1,\,\cdots,\,n+m\}$ satisfying $n-1, n+m+1\notin E$, $E_3$ is an infinitely set like $\{i_0,\,i_0+1,\,i_0+2,\,\cdots\}$ satisfying $i_0-1\notin E$ and $E_4$ is an infinitely set like $\{j_0,\,j_0-1,\,j_0-2,\,\cdots\}$ satisfying $j_0+1\notin E$, $n, m, i_0, j_0\in\Z$. Then we only need to decompose $E_1$, $E_2$, $E_3$ and $E_4$ separately.

For $E_1$ set, we obtain that $E_1=\{i_0\}=S_{i_0, 0}$ and $4S_{i_0,0}=\{i_0-3,\,i_0-2,\,i_0-1,\,i_0,\,i_0+1,\,i_0+2,\,i_0+3\}$, then it's obvious that $4S_{i_0,0}\cap {E}^c\neq\varnothing$.

For $E_2$ set, we consider the following two cases.

Case I: If $|E_2|=m+1$ is odd, then the set itself is an interval of symmetric integers, denoted as $S_{m,N}$ and obviously, $4S_{m,N}\cap {E}^c\neq\varnothing$.

Case II: If $|E_2|=m+1$ is even, then there exist two odd numbers $k_1$ and $k_2$ such that $k_1+k_2=|E_2|$ and $E_2$ can be splitted  into two symmetric integer intervals $E_{2,1}$ and $E_{2,2}$ with $|E_{2,1}|=k_1$, $|E_{2,2}|=k_2$. Then repeat case I.

For $E_3$ set,  we can decompose this set into a union of disjoint symmetric intervals of cardinalities $2N_j+1$, $j\in\Z_+$, i.e.,
\begin{equation*}
E_3:=\mathop{\bigcup}\limits_{j\in\Z_+}S_{m_j,N_j}=\mathop{\bigcup}\limits_{j\in\Z_+}\{m_j-N_j,\,\cdots,\,m_j,\,\cdots,\,m_j+N_j\},
\end{equation*}
where
\begin{equation*}
\begin{aligned}
&m_j:=i_0+2\sum\limits_{i=1}^{j-1}2^{j-i}+(j-1)+2^j=i_0+3\cdot 2^j+j-5;\\
&N_j:=2^j;\\
&d_j:=|m_j-i_0|=2\sum\limits_{i=1}^{j-1}2^{j-i}+(j-1)+2^j=3\cdot 2^j+j-5;\\
&4N_j-d_j=2^j-j+5=2^j-(j-1)+4>0.
\end{aligned}
\end{equation*}
This implies that
\begin{equation*}
E_3=\mathop{\bigcup}\limits_jS_{m_j,N_j}\quad\text{and}\quad 4S_{m_j,N_j}\bigcap {E}^c\neq\varnothing,\quad j\in\Z_+.
\end{equation*}

The decompose for $E_4$ set is similar to that for $E_3$ set, the details being omitted. We finish the proof of Lemma \ref{l3.2}.
\end{proof}

\begin{lemma}\label{l3.16}
If $0<\alpha<1$, there exists a positive constant $K$ depending only on $\alpha$ such that if $a>0, b\ge 6, c>0$, $x$ is nonnegative sequence, $S\subset\Z$ is a symmetric interval such that $I_\alpha x(k)\le a$ at some point of $S$ and $E:=\{k\in S:I_\alpha x(k)>ab,\M_\alpha x(k)\le ac\}$ is the subset of $S$, then $\left|E\right|\le K\left|S\right|\left(\frac{c}{b}\right)^\frac{1}{1-\alpha}$.
\end{lemma}

\begin{proof}
Let $g(k):=x(k)$ on $2S$ and 0 elsewhere. Set $h(k):=x(k)-g(k)$ and there exists a $t\in S$ such that $\M_\alpha x(t)\le ac$. By Lemma \ref{c3.5}$\rm(ii)$, for any positive $a$ and $b$, we have
\begin{equation*}\label{eq3.16}
\left|\left\{k\in\Z: I_\alpha g(k)>\frac{ab}{2}\right\}\right|\le C\left(\frac{1}{ab}\sum_{k\in \Z}g(k)\right)^{\frac{1}{1-\alpha}}.
\end{equation*}
Let $S:=S_{m,N}$, $m\in\Z, N\in\N$, $J$ be the symmetric interval centered at $t$ and three times as long as those of $S$. Then we have $2S\subset J$, $J\subset 4S$ and
\begin{equation*}
\sum_{k\in\Z}g(k)\le \sum_{k\in J}x(k)\le \M_\alpha (t)\left|J\right|^{1-\alpha}\le ac\left|4S\right|^{1-\alpha}=ac(8N+1)^{1-\alpha}\le 4^{1-\alpha}ac\left|S\right|^{1-\alpha}.
\end{equation*}
From this, it follows that
\begin{equation}\label{eq3.17}
\left|\left\{k\in\Z: I_\alpha g(k)>\frac{ab}{2}\right\}\right|\le C\left|S\right|\left(\frac{c}{b}\right)^\frac{1}{1-\alpha}.
\end{equation}
Now let $s$ be the point of $S$ such that $|I_\alpha x(s)|\le a$. If $k\in S$ and $i\notin 2S$, then
\begin{equation*}
\left|s-i\right|\le |s-k|+|k-i|\le 2|k-i|+\left|k-i\right|\le 3|k-i|=:L|k-i|.
\end{equation*}
Therefore, for every $k\in S$, we have
\begin{equation*}
\begin{aligned}
I_\alpha h(k)=\sum\limits_{i\neq k}\frac{h(i)}{\left|k-i\right|^{1-\alpha}}
\le L^{1-\alpha}\sum\limits_{i\neq s}\frac{h(i)}{\left|s-i\right|^{1-\alpha}}
\le LI_\alpha x(s)
\le La.
\end{aligned}
\end{equation*}
Let $B=2L$. If $b\ge 6$, then for any $k\in S$, we obtain $I_\alpha h(k)\le \frac{ab}{2}$.
Thus, for any $k\in E$, we have
\begin{equation*}
I_\alpha g(k)=\sum\limits_{i\neq k}\frac{x(i)-h(i)}{|i-k|^{1-\alpha}}=I_\alpha x(k)-I_\alpha h(k)>ab-\frac{ab}{2}=\frac{ab}{2},
\end{equation*}
which together with (\ref{eq3.17}) implies that
\begin{equation*}
\left|E\right|\le\left|\left\{k\in\Z: I_\alpha g(k)>\frac{ab}{2}\right\}\right|\le C\left|S\right|\left(\frac{c}{b}\right)^\frac{1}{1-\alpha}=:K\left|S\right|\left(\frac{c}{b}\right)^\frac{1}{1-\alpha}.
\end{equation*}
We finish the proof of Lemma \ref{l3.16}.
\end{proof}

The next lemma will reveal another relation between discrete Riesz potentials $I_\alpha$ and discrete fractional maximal operator $\M_\alpha$. Using Lemmas \ref{l3.2}, \ref{l3.16} and \ref{l3.11}, following lemma can be proved. Continuous version of Lemma \ref{l3.12}
can be found in \cite{1974m} and for the sake of the completeness, the proofs are included.

\begin{lemma}\label{l3.12}
Let $0<\alpha<1$ and $0< q<\infty$. If $\omiga\in\huaA_\infty$ and $x=\{x(k)\}_{k\in\Z}\subset\R$, then there exists a positive constant C such that
\begin{equation*}
\sum_{k\in\Z}\left|I_\alpha x(k)\right|^q\omiga(k)\le C\sum_{k\in\Z}|\M_\alpha x(k)|^q\omiga(k).
\end{equation*}
\end{lemma}

\begin{proof}
We may assume that $x=\{x(k)\}_{k\in\Z}$ is nonnegative and supported in some symmetric interval $S$ or else replacing $x$ by $|x|=\{|x(k)|\kaf_{\{|k|\le M\}}(k)\}_{k\in\Z}, M\in\Z_+$. Given $a>0$, by Lemma \ref{l3.2}, we can decompose the set $\{k\in\Z:I_\alpha x(k)>a\}$ into a list of disjoint symmetric integers intervals $\{S_{m_j,N_j}\}_j$ and for every $j$, we have $I_\alpha x(k)\le a$ at some point $k$ of $4S_{m_j,N_j}$.

Let $0<\alpha<1$. For $4S_{m_j,N_j}$, by Lemma \ref{l3.16}, there exists $K>0$ depending on $\alpha$ such that for any $a>0$, $b\ge 6$, and $c>0$, and $E_j:=\{k\in 4S_{m_j,N_j}: I_\alpha x(k)>ab, \M_\alpha x(k)\le ac\}$, we have
\begin{equation}\label{eq2.8}
|E_j|\le K|4S_{m_j,N_j}|\left(\frac{c}{b}\right)^{\frac{1}{1-\alpha}}.
\end{equation}
Besides, for $\omiga\in\huaA_\infty$, by Proposition \ref{m2.11} with $\epsilon=\frac{b^{-q}}{2}$, there exists $\delta>0$ such that for any subset $S\subset\text{interval}\, J\subset\Z$ with $|S|\le\delta |J|$, it holds true that $\omiga(S)\le\epsilon \omiga(J)$. Choosing $D>0$ such that $\delta=4K\left(\frac{D}{b}\right)^{\frac{1}{1-\alpha}}$, and let $c\in (0,D]$. From this and (\ref{eq2.8}), it follows that


\begin{equation*}
\left|E_j\right|\le K\left|4S_{m_j,N_j}\right|\left(\frac{c}{b}\right)^\frac{1}{1-\alpha}<\delta \left|S_{m_j,N_j}\right|,
\end{equation*}
which together with Proposition \ref{m2.11} implies that
\begin{equation}\label{eq1.1}
\omiga(E_j)\le \frac{b^{-q}}{2}\omiga(S_{m_j,N_j}).
\end{equation}
Let
\begin{equation*}
\begin{aligned}
A_1:&=\{k\in\Z:I_\alpha x(k)>ab\},\\
A_2:&=\{k\in\Z:I_\alpha x(k)>a\},\\
B_1:&=\{k\in\Z:\M_\alpha x(k)\le ac\}.
\end{aligned}
\end{equation*}
Since $A_1\cap B_1\subset A_1\subset A_2$, for any $j$, by (\ref{eq1.1}), we have
\begin{equation*}
\omiga(A_1\cap B_1)\le \frac{b^{-q}}{2}\omiga(A_2).
\end{equation*}
This implies that
\begin{equation}\label{eq1.2}
\begin{aligned}
\omiga(A_1)&=\omiga(A_1\cap B_1)+\omiga(A_1\cap (B_1)^c)\le\omiga(A_1\cap B_1)+\omiga((B_1)^c)\\
&\le \frac{b^{-q}}{2}\omiga(A_2)+\omiga((B_1)^c).
\end{aligned}
\end{equation}

Given $k\notin 3S$, let $s$ be the point in $S$ closest to $k$ and let $J$ be the smallest symmetric interval with center at $k$ and contains $S$. Then we have $|J|\le 4|k-s|$ and
\begin{equation*}
I_\alpha x(k)=\sum\limits_{i\in S}\frac{x(i)}{|i-k|^{1-\alpha}}\le \left|k-s\right|^{\alpha-1}\sum_{i\in S}x(i)\le \frac{\left|J\right|^{1-\alpha}}{\left|k-s\right|^{1-\alpha}}\M_\alpha x(k)\le 4^{1-\alpha}\M_\alpha x(k).
\end{equation*}
Define $c=\min\left(D,\frac{1}{4^{1-\alpha}}\right)$, then we have
\begin{equation}\label{eq3.19}
\{k\in\Z:I_\alpha x(k)>a\}\cap(3S)^c\subset\{k\in\Z:\M_\alpha x(k)>ac\}.
\end{equation}
From (\ref{eq1.2}) and (\ref{eq3.19}), it follows that
\begin{equation}\label{eq4.1}
\begin{aligned}
\omiga(A_1)&\le\omiga((B_1)^c)+\frac{b^{-q}}{2}\omiga(A_2)=\omiga((B_1)^c)+\frac{b^{-q}}{2}\omiga(A_2\cap 3S)+\frac{b^{-q}}{2}\omiga(A_2\cap (3S)^c)\\
&\le\omiga((B_1)^c)+\frac{b^{-q}}{2}\omiga(A_2\cap 3S)+\frac{b^{-q}}{2}\omiga((B_1)^c)\\
&\le 2\omiga((B_1)^c)+\frac{b^{-q}}{2}\omiga(A_2\cap 3S).
\end{aligned}
\end{equation}

Next multiply both sides of (\ref{eq4.1}) by $a^{q-1}$ and integrate $a$ from 0 to some positive integer $N$. After a change of variable the left side becomes
\begin{equation}\label{eq3.21}
b^{-q}\int^{bN}_0a^{q-1}\omiga(\{k\in\Z:I_\alpha x(k)>a\})da.
\end{equation}
Similarly, with a change of variables for the first integral on right, the right side becomes
\begin{equation}\label{eq3.22}
2c^{-q}\int^{Nc}_0 a^{q-1}\sum_{\{k\in\Z:\M_\alpha x(k)>a\}}\omiga(k)da+\frac{b^{-q}}{2}\int^N_0 a^{q-1}\sum_{\{k\in\Z: I_\alpha x(k)>a\}\cap 3S}\omiga(k)da.
\end{equation}

Since $\omiga(\{k\in\Z: I_\alpha x(k)>a\}\cap 3S)\le \omiga(3S)<+\infty$, then the second term in (\ref{eq3.22}) is finite and
\begin{equation*}
\int^N_0 a^{q-1}\omiga(\{k\in\Z: I_\alpha x(k)>a\}\cap 3S)da\le\int^{bN}_0a^{q-1}\omiga(\{k\in\Z:I_\alpha x(k)>a\})da.
\end{equation*}
Therefore,
\begin{equation}\label{eq3.23}
\frac{1}{2}b^{-q}\int^{bN}_0 a^{q-1}\omiga(\{k\in\Z:I_\alpha x(k)>a\})da\le 2c^{-q}\int^{Nc}_0 a^{q-1}\omiga(\{k\in\Z:\M_\alpha x(k)>a\})da.
\end{equation}
Letting $N\rightarrow\infty$ on both sides of (\ref{eq3.23}) and using Lemma \ref{l3.11}, it implies that
\begin{equation}\label{eq3.24}
\frac{b^{-q}}{2}\sum_{k\in\Z}\left|I_\alpha x(k)\right|^q\omiga(k)\le 2c^{-q}\sum_{k\in\Z}[\M_\alpha x(k)]^q\omiga(k).
\end{equation}

If $x$ doesn't have compact support, applying $|x|=\{|x(k)|\kaf_{\{|k|\le M\}}(k)\}_{k\in\Z}$ to all of the above processes,
\begin{equation*}
\begin{aligned}
\frac{b^{-q}}{2}\sum_{k\in\Z}[I_\alpha (|x|\kaf_{\{|k|\le M\}})(k)]^q\omiga(k)&\le 2c^{-q}\sum_{k\in\Z}[\M_\alpha (x\kaf_{\{|k|\le M\}})(k)]^q\omiga(k)\\
&\le 2c^{-q}\sum_{k\in\Z}[\M_\alpha x(k)]^q\omiga(k).
\end{aligned}
\end{equation*}
Taking $M\rightarrow\infty$ on above inequality and using monotone convergence theorem, we finish the proof of Lemma \ref{l3.12}.
\end{proof}

\begin{proof}[Proof of Theorem \ref{t1.1}]
For $\omiga\in\huaA(p,q)$, by Proposition \ref{m2.13}$\rm(i)$ and the definition of $\huaA_\infty$, we have $\omiga(k)^q\in\huaA_{1+\frac{q}{p'}}\subset\huaA_\infty$. Thus, from Lemma \ref{l3.12} and Theorem \ref{t3.1}, it follows that (\ref{eq1.3}) holds true.
\end{proof}

\subsection{Estimates for discrete Riesz potentials on discrete weighted Morrey spaces \label{s4.2}}
\begin{theorem}\label{t3.10}
Let $0<\alpha<1$, $1<p<\frac{1}{\alpha}$, $\frac{1}{q}=\frac{1}{p}-\alpha$ and $q<2p$. Set $s=\frac{qp}{2p-q}$ and $\omiga\in\huaA(p,q)$. Then $I_\alpha\in {l{^q_s}}{(\omiga^q)}$ for every $x\in l{^p_q}(\omiga^p,\omiga^q)$ and there exists a positive constant $C$ such that
\begin{equation*}
\|I_\alpha x\|_{l{^q_s}(\omiga^q)}\le C\|x\|_{l{^p_q}(\omiga^p,\omiga^q)}.
\end{equation*}
\end{theorem}

The next lemma plays an important role in our proof of Theorem \ref{t3.10}. We say that $\omiga$ satisfies the {\it reverse doubling condition} if $\omiga$ has the property (\ref{eq3.26}) of the following lemma.

\begin{lemma}\label{l3.14}
Let $1\le p<\infty$ and $\omiga\in\huaA_p$.
\begin{itemize}
\item[\rm(i)]\cite[Proposition 2.9]{hao}
Then there exists a positive constant $C>0$ such that
\begin{equation*}
\omiga(\lambda S_{m,N})\le \left(\frac{3}{2}C\right)^p\lambda^p\omiga(S_{m,N})\quad\text{and}\quad\omiga(nLI)\le Cn^p\omiga(I),
\end{equation*}
where $I:=\{(j-1)2^N+1,\cdots,j2^N\}$, $nLI:=\{(j-n)2^N+1,\cdots,j2^N\}$, $j\in\Z$ and $N\in\N$.
\item[\rm(ii)]
There exists a positive constant $C_1>1$ such that
\begin{equation}\label{eq3.26}
\omiga(2S_{m,N})\ge C_1\omiga(S_{m,N}),\,m\in\Z,\, N\in\N.
\end{equation}
\end{itemize}
\end{lemma}

\begin{proof}
$\rm(ii)$ When $N\ge 3$, we fix a symmetric interval with center $m$ and cardinality $2N+1$, i.e.,  $S_{m,N}=\{k:\left|k-m\right|\le N\}$. Then we can choose a symmetric interval $J\subset 2S_{m,N}$ with side length $\lfloor \frac{N-1}{2}\rfloor$ which is disjoint from the $S_{m,N}$ and
\begin{equation*}
\omiga(S_{m,N})+\omiga(J)\le \omiga(2S_{m,N}).
\end{equation*}
On the other hand, since $S_{m,N}\subset 10J$, by Lemma \ref{l3.14}(i), there exists a positive constant $C$ such that $\omiga(S_{m,N})\le \omiga(10J)\le C(15)^p\omiga(J)$, thus we have
\begin{equation*}
\omiga(S_{m,N})+\frac{\omiga(S_{m,N})}{C(15)^p}\le \omiga(2S_{m,N}).
\end{equation*}

When $0\le N\le 2$, for fixed symmetric interval $S_{m,N}=\{k:\left|k-m\right|\le N\}$, we can choose a interval  $J=\{m+N+1\}\subset 2S_{m,N}$ which is disjoint from the $S_{m,N}$ and
\begin{equation}\label{eq4.2}
\omiga(S_{m,N})+\omiga(J)\le \omiga(2S_{m,N}).
\end{equation}
On the other hand, we perform left dilation of the $J$, i.e.,
\begin{equation*}
6LJ:=\{m+N-4,\cdots,m+N+1\},
\end{equation*}
then we have $S_{m,N}\subset 6LJ$, which together with Lemma \ref{l3.14}(i) implies that
\begin{equation*}
\omiga(S_{m,N})\le \omiga(6LJ)\le C6^p\omiga(J),
\end{equation*}
where C is a positive constant.
From this and (\ref{eq4.2}), it follows that
\begin{equation*}
\omiga(S_{m,N})+\frac{\omiga(S_{m,N})}{C(6)^p}\le \omiga(2S_{m,N}).
\end{equation*}
Combining the above estimates, we finish the proof of Lemma \ref{l3.14}.
\end{proof}

Next we shall prove the Theorem \ref{t3.10}.
\begin{proof}[Proof of Theorem \ref{t3.10}]
We may assume that $x$ is a nonnegative sequence or else replace $x$ by $|x|=\{|x(k)|\}_{k\in\Z}$. Fix a symmetric interval $S_{m,N}$ and decompose $x=x_1+x_2$ with $x_1=x\kaf_{2S_{m,N}}$, we obtain
\begin{small}
\begin{equation*}
\begin{aligned}
&\left(\sum\limits_{k\in S_{m,N}}\omiga(k)^q\right)^{1-\frac{q}{p}}\sum\limits_{k\in S_{m,N}}\left|I_\alpha x(k)\right|^q\omiga(k)^q\\
\le& C\left(\sum\limits_{k\in S_{m,N}}\omiga(k)^q\right)^{1-\frac{q}{p}}\sum\limits_{k\in S_{m,N}}\left(\left|I_\alpha x_1(k)\right|^q+\left|I_\alpha x_2(k)\right|^q\right)\omiga(k)^q\\
=&C\left(\sum\limits_{k\in S_{m,N}}\omiga(k)^q\right)^{1-\frac{q}{p}}\sum\limits_{k\in S_{m,N}}|I_\alpha x_1(k)|^q\omiga(k)^q+C\left(\sum\limits_{k\in S_{m,N}}\omiga(k)^q\right)^{1-\frac{q}{p}}\sum\limits_{k\in S_{m,N}}|I_\alpha x_2(k)|^q\omiga(k)^q\\
=:&\rm{I}+\rm{II}
\end{aligned}
\end{equation*}
\end{small}

To estimate term $\rm I$, using the fact that $I_\alpha$ is bounded from $l^p_{\omiga^p}$ to $l^q_{\omiga^q}$ with $\omiga\in\huaA(p,q)$ (see Theorem \ref{t1.1}), we have
\begin{align*}
\sum\limits_{k\in S_{m,N}}\left|I_\alpha x_1(k)\right|^q\omiga(k)^q &\le \sum\limits_{k\in\Z}\left|I_\alpha x_1(k)\right|^q\omiga(k)^q\le \left(\sum\limits_{k\in\Z}\left|x_1(k)\right|^p\omiga(k)^p\right)^\frac{q}{p}\\
&= \left(\sum\limits_{k\in S_{m,N}}\omiga(k)^q\right)^{1-\frac{q}{p}}\left(\sum\limits_{k\in S_{m,N}}\left|x(k)\right|^p\omiga(k)^p\right)^\frac{q}{p}\left(\sum\limits_{k\in S_{m,N}}\omiga(k)^q\right)^{\frac{q}{p}-1}\\
&\le C\|x\|^q_{{l^p_q}_{(\omiga^p,\omiga^q)}}\left(\sum\limits_{k\in S_{m,N}}\omiga(k)^q\right)^{\frac{q}{p}-1}.
\end{align*}
It implies that
$
\|I_\alpha x_1\|_{{l^q_\frac{qp}{2p-q}}_{(\omiga^q)}}\le C\|x\|_{{l^p_q}_{(\omiga^p,\omiga^q)}}.
$

To estimate term $\rm II$, for every $k\in S_{m,N}$ and $i\in(2S_{m,N})^c$, we have $|m-k|\le N$, $|k-i|\ge N+1$ and
$|m-i|\le|m-k|+|k-i|\le 2|k-i|$. Thus we obtain
\begin{equation*}
\left|I_\alpha x_2(k)\right|\le\sum_{i\in\Z\setminus{\{k\}}}\frac{\left|x_2(i)\right|}{\left|i-k\right|^{1-\alpha}}\le 2^{1-\alpha}\sum_{\left|i-m\right|>2N}
\frac{\left|x(i)\right|}{\left|m-i\right|^{1-\alpha}}.
\end{equation*}
Then
\begin{small}
\begin{equation*}
{\rm II}\le 2^{1-\alpha}\left(\sum\limits_{k\in S_{m,N}}\omiga(k)^q\right)^{2-\frac{q}{p}}\left(\sum\limits_{\left|i-m\right|>2N}
\frac{\left|x(i)\right|}{\left|i-m\right|^{1-\alpha}}\right)^q.
\end{equation*}
\end{small}
From discrete H\"older's inequality, $\omiga\in\huaA(p,q)$ and $\frac{1}{q}=\frac{1}{p}-\alpha$, it follows that
\begin{small}
\begin{align*}
&\sum\limits_{\left|m-i\right|>2N}\frac{\left|x(i)\right|}{\left|m-i\right|^{1-\alpha}}\\
\quad &=\sum\limits^\infty_{j=1}\sum\limits_{2^jN<\left|m-i\right|\le 2^{j+1}N}\frac{\left|x(i)\right|}{\left|m-i\right|^{1-\alpha}}\\
&\le\sum\limits^\infty_{j=1}\frac{1}{(2^jN)^{1-\alpha}}\left(\sum\limits_{\left|m-i\right|\le 2^{j+1}N}\left|x(i)\right|^p\omiga(i)^p\right)^\frac{1}{p}\left(\sum\limits_{\left|m-i\right|\le 2^{j+1}N}\omiga(i)^{-p'}\right)^\frac{1}{p'}\\
&\le C\sum\limits^\infty_{j=1}\frac{1}{(2^jN)^{1-\alpha}}\left(\sum\limits_{\left|m-i\right|\le 2^{j+1}N}\left|x(i)\right|^p\omiga(i)^p\right)^\pfzy\left(\sum\limits_{\left|m-i\right|\le 2^{j+1}N}\omiga(i)^{q}\right)^{-\frac{1}{q}}\left|S_{m,2^{j+1}N}\right|^{\frac{1}{q}+\frac{1}{p'}}\\
&\le C\sum\limits^\infty_{j=1}(2^jN)^{\alpha-1}\|x\|_{{l^p_q}_{(\omiga^p,\omiga^q)}}\left(\sum\limits_{\left|m-i\right|\le 2^{j+1}N}\omiga(i)^q\right)^{\frac{1}{p}-\frac{2}{q}}(2\cdot 2^{j+1}N+1)^{\frac{1}{q}+\frac{1}{p'}}\\
&\le C\sum\limits^\infty_{j=1}(2^jN)^{\alpha-1}(3\cdot 2^{j+1}N)^{\frac{1}{q}+\frac{1}{p'}}\|x\|_{{l^p_q}_{(\omiga^p,\omiga^q)}}\left(\sum\limits_{\left|m-i\right|\le 2^{j+1}N}\omiga(i)^q\right)^{\frac{1}{p}-\frac{2}{q}}\\
&\le C6^{\frac{1}{q}+\frac{1}{p'}}\sum\limits^\infty_{j=1}(2^jN)^{\alpha-1+\frac{1}{q}+\frac{1}{p'}}\|x\|_{{l^p_q}_{(\omiga^p,\omiga^q)}}\left(\sum\limits_{\left|m-i\right|\le 2^{j+1}N}\omiga(i)^q\right)^{\frac{1}{p}-\frac{2}{q}}\\
&\le C\|x\|_{{l^p_q}_{(\omiga^p,\omiga^q)}}\sum\limits^\infty_{j=1}\left(\sum\limits_{\left|m-i\right|\le 2^{j+1}N}\omiga(i)^q\right)^{\frac{1}{p}-\frac{2}{q}}.
\end{align*}
\end{small}
By Lemma \ref{l3.14} and $q<2p$, we have
\begin{small}
\begin{equation*}
\begin{aligned}
&\left(\sum\limits_{k\in S_{m,N}}\omiga(k)^q\right)^{2-\frac{q}{p}}\left(\sum\limits_{\left|i-m\right|>2N}
\frac{\left|x(i)\right|}{\left|m-i\right|^{1-\alpha}}\right)^q\\
\le & C\|x\|^q_{{l^p_q}_{(\omiga^p,\omiga^q)}}\left(\sum\limits^\infty_{j=1}\frac{\omiga^q(S_{m,N})^{\frac{2}{q}-\frac{1}{p}}}
{\omiga^q(S_{m,2^{j+1}N})^{\frac{2}{q}-\frac{1}{p}}}\right)^q
\le C\|x\|^q_{{l^p_q}_{(\omiga^p,\omiga^q)}}\left(\sum\limits^\infty_{j=1}({C_1}^{j+1})^{\frac{1}{p}-\frac{2}{q}}\right)^q
\le C\|x\|^q_{{l^p_q}_{(\omiga^p,\omiga^q)}}.
\end{aligned}
\end{equation*}
\end{small}
Combining the above estimates, we finish the proof of Theorem \ref{t3.10}.
\end{proof}

\begin{theorem}\label{t3.11}
Let $0<\alpha<1$, $1<p<\frac{1}{\alpha}$, $\frac{1}{q}=\frac{1}{p}-\alpha$ and $q<2p$. Set $s=\frac{qp}{2p-q}$ and $\omiga\in\huaA(p,q)$. Then the discrete fractional maximal operator $\M_\alpha$ is bounded from ${l^p_q}_{(\omiga^p,\omiga^q)}$ to ${l^q_s}_{(\omiga^q)}$
\end{theorem}

\begin{proof}
The Theorem \ref{t3.11} immediately follows from the pointwise inequality $\M_\alpha x(k)\le I_\alpha(\left|x\right|)(k)$ for $0<\alpha<1$ (see Lemma \ref{t3.6}) and Theorem \ref{t3.10}.
\end{proof}

\bigskip\medskip

\noindent Xuebing Hao, Shuai Yang and Baode Li (Corresponding author),
\medskip

\noindent College of Mathematics and System Sciences\\
Xinjiang University\\
Urumqi, 830017\\
P. R. China
\smallskip

\noindent{E-mail }:\\
\texttt{1659230998@qq.com} (Xuebing Hao)\\
\texttt{2283721784@qq.com} (Shuai Yang)\\
\texttt{baodeli@xju.edu.cn} (Baode Li)\\
\bigskip \medskip

\end{document}